\documentclass{amsart}

\usepackage{amsmath,amsthm,amssymb,pifont,colortbl,amscd, wrapfig,extpfeil}
\usepackage{bigdelim,multirow}
\usepackage{mathrsfs}
\usepackage[all]{xy}
 \usepackage{graphicx}
\usepackage{pgf,tikz}
\usetikzlibrary{decorations}
\usetikzlibrary{shapes,arrows}
\usetikzlibrary{calc,through,intersections,snakes,backgrounds}
\newtheorem{theorem}{Theorem}[subsection]
\newtheorem{definition}[theorem]{Definition}
\newtheorem{lemma}[theorem]{Lemma}
\newtheorem{proposition}[theorem]{Proposition}
\newtheorem{example}[theorem]{Example}
\newtheorem{corollary}[theorem]{Corollary}
\newtheorem{remark}[theorem]{Remark}

\def\mod{\mathop{\mathrm{mod}}\nolimits}

\def\gr{\mathop{\mathrm{gr}}\nolimits}
\def\Hom{\mathop{\mathrm{Hom}}\nolimits}
\def\Ker{\mathop{\mathrm{Ker}}\nolimits}
\def\Aut{\mathop{\mathrm{Aut}}\nolimits}

\def\Der{\mathop{\mathrm{Der}}\nolimits}

\def\comm{\mathop{\mathrm{comm}}\nolimits}
\def\col{\mathop{\mathrm{col}}\nolimits}

\numberwithin{equation}{section}

\begin{document}
\title{Group-like expansions and Invariants of string links}
\author{Hisatoshi Kodani}
\address{Faculty of Mathematics, Kyushu university, 744, Motooka, Nishi-ku, Fukuoka, 819-0395 Japan.}
\email{h-kodani@math.kyushu-u.ac.jp}


\begin{abstract}
In this article, we define and study the total Milnor invariant and the infinitesimal Morita-Milnor homomorphism as punctured disk analogues of the total Johnson map and the infinitesimal Morita homomorphism studied by Kawazumi and Massuyeau in the case of  surface of positive genus with one boundary component.
\end{abstract}

\maketitle

\section*{Introduction}
Special derivations of the free Lie algebra was introduced by Ihara for the sake of studying the structure of the absolute Galois group of rational number field in the context of the Galois action on the algebraic fundamental group of projective line minus three points (\cite{Ih}).

Special derivations are also studied by A. Alekseev, B. Enriquez,  C. Torossian in the context of Kashiwara-Vergne conjecture and they showed that the Kontsevich integral defines the special type of expansion (\cite{AET}). Then, Massuyeau introduced the notion of special expansions which generalize this type of expansions (\cite{MA2}).  
 
 In this article, we firstly define the total Milnor invariant for special derivations of the free Lie algebra. Then, we define the total Milnor invariant for string links in terms of special expansions. This total Milnor invariant is aimed for the punctured disk counter part of the total Johnson map introduced by Kawazumi (\cite{Ka}) for the case of a surface of positive genus with one boundary component. We show that total Milnor invariant for the string links gives the generalization of Milnor invariant of degree $k$ which has diagrammatic interpretation and appear in tree reduction of Kontsevich integral (\cite{HabeM}).
 
Secondly, we define the infinitesimal Morita-Milnor homomorphism as a punctured disk analogue of the infinitesimal Morita homomorohism introduced by Massuyeau for the case of a surface of positive genus with one boundary component. We also show that this infinitesimal Morita-Milnor homomorohism has similar properties as the case of the infinitesimal Morita homomorohism.

To construct the punctured disk analogue of invariant of a surface of positive genus with one boundary component may be useful not only to study string links but also to study the absolute Galois group because they have similar action on the (completed) free group (cf. \cite{KMT}).

The construction of present article is as follows. In \S1, we recall the basics on the pure braid group, string links and the Milnor invariant of degree $k$. In \S2, we recall the diagrammatic interpretation of the third homology group of the free Lie algebra studied by Igusa-Orr and Massuyeau. In \S3, we firstly recall the Malcev completion of a group. Then, we recall the notion of group-like expansions and special expansions introduced by Massuyeau. In \S3, we define the total Milnor invariant for special derivations of free Lie algebra. Then, we define the total Milnor invariant for the monoid of string links in terms of a special expansion and we show some properties of it. \S5, we define the infinitesimal Morita-Milnor homomorohism and we show its relation with the total Milnor invariant defined in \S4.

\subsection*{Notation} For a group $G$, we denote by $\Gamma_m G$ the $m$-th term of the lower central series of $G$ defined by
$$
\Gamma_1G:=G, \ \Gamma_m G:=[\Gamma_{m-1} G, G]\ (m \geqslant 2),
$$
and we set
$$
\gr_{m}G := \Gamma_{m}G/\Gamma_{m+1}G.
$$
The commutator of $a, b \in G$ is defined by $[a,b]:=aba^{-1}b^{-1}$. Let $\mathbb{K}$ be a commutative field of characteristic 0.

\section{The pure braid group and string links}

In this section, we recall the basics of pure braid group and string links. Then, we recall the Milnor invariant of degree $k$. For more details on these subject, see \cite{B} and \cite{HL}.
\subsection{The pure braid group and string links}\label{sec:pure}
Let $n \geqslant 2$ be a fixed integer. Let $PB_n$ be the pure braid group with $n$ strings. It is well known that $PB_n$ is generated by $A_{ij}$ $(1\leqslant i <j \leqslant n)$ subject to the following relation:

\begin{eqnarray*}
A_{rs}A_{ij}A_{rs}^{-1}=\left\{
\begin{aligned}
&A_{ij}&&(\text{if}\ s<i\ \text{or}\ i<r<s<j),\\
&A_{rj}^{-1}A_{ij}A_{rj}&&(\text{if}\ s=i),\\
&A_{rj}^{-1}A_{sj}^{-1}A_{ij}A_{sj}A_{rj}&&(\text{if}\ i=r<s<j),\\
&A_{rj}^{-1}A_{sj}^{-1}A_{rj}A_{sj}A_{ij}A_{sj}^{-1}A_{rj}^{-1}A_{sj}A_{rj}&&(\text{if}\ r<i<s<j).
\end{aligned}
\right.
\end{eqnarray*}
Here, the product $L\cdot L'$ is defined by stacking $L'$ on $L$ for each $L, L' \in PB_n$.
Each generator 

\begin{remark}\label{rm:mapping}{\rm 
The pure braid group $PB_n$ is the pure mapping class group of a $n$-punctured disk $D_n$, i.e. $PB_n$ is the kernel of $\nu : \mathcal{M}(D_n) \rightarrow S_n$ where $\mathcal{M}(D_n)$ is the mapping class group\footnote{Here, the mapping class group of a surface $\Sigma$ is defined as the group of isotopy class of orientation preserving homeomorphisms of $\Sigma$ which fix the boundary pointwise.} of $D_n$ and $S_n$ is the symmetric group of degree $n$.

}
\end{remark}
Note that the fundamental group $\pi_1(D_n)$ may be identified with the free group $F_n$ on $x_1,\ldots, x_n$ where $x_i$ represents a small loop around $i$-th puncture clockwise. Since $P_n$ is the pure mapping class group $\mathcal{PM}(D_n)$ of $n$-punctured disk $D_n$ (Remark \ref{rm:mapping}), $PB_n$ acts on the fundamental group $\pi_1(D_n)=F_n$ naturally. Then, we obtain the homomorphism, called the {\it Artin representation}, 
$$
Art: PB_n \overset{\simeq}{\longrightarrow} \Aut_0(F_n).
$$
Here,  the subgroup $\Aut_0(F_n)$ consists of automorphisms $\varphi$ such that $\varphi(x_i)=y_ix_iy_i$ $(1\leqslant i \leqslant n)$ for some $y_i \in F_n$ and $\varphi(x_1\cdots x_n)=x_1\cdots x_n$.
\begin{remark}\label{longi}{\rm
(1) For each $L \in PB_n$, we have the words $y_i$ by $Art(L)(x_i)=y_ix_iy_i^{-1}$. These words are uniquely determined under the  condition of the coefficient of $x_i$ in $[y_i] \in F_n/\Gamma_2 F_n$  to be 0. Hence, $Art(L)$ is completely determined by $\vec{y}(L)=(y_1(L),\ldots, y_n(L))$ under this condition. These words $y_i$ is coincides with the word of longitude by meridians of the link obtained by closing the pure braid $L$ \footnote{Here, this condition means that we restrict ourselves for the case that pure braids are 0-framed pure braids.}.
 }
\end{remark}
Let $SL_n$ be the monoid of the isotopy classes of string links. As is explained in \cite{HabeM}, the pure braid group $PB_n$ is identified with the group of invertible elements of $SL_n$. By Stalling's theorem, we have the $k$-th Artin representation, for each integer $k \geqslant 1$,
$$
Art_k: SL_n \longrightarrow \Aut_0(F_n/ \Gamma_{k+1}F_n).
$$
Here, $\Aut_0(F_n/ \Gamma_{k+1}F_n)$ denotes the subgroup of $\Aut(F_n/\Gamma_{k+1}F_n)$ consisting of automorphisms which send class of generator $\{x_i\}$ $(1\leqslant i \leqslant n)$ to its conjugate and fix the class $\{x_1\cdots x_n\}$. Note that Remark \ref{longi} (1) also holds for the $k$-th Artin representation of string links, i.e. the $k$-th Artin representation $Art_k(L)$ is completely determined by $\vec{y}{}^{(k)}(L)=(y_1^{(k)}(L), \ldots, y_n^{(k)}(L))$ and $y_i^{(k)}(L)$ is a word of longitudes by meridians.

\subsection{The Milnor invariant of degree $k$} 
Notations being as in \S \ref{sec:pure}, we denote by the kernel of $Art_k$ by $SL_n(k)$, i.e.
\begin{eqnarray*}
&SL_n(k)&:=\Ker(Art_k)\\
&\ &=\{ L \in SL_n \mid Art_k(L)(g)g^{-1} \in \Gamma_{k+1}\}\ (k \geqslant 1).
\end{eqnarray*}
We then have the descending series
$$
SL_n=SL_n(1) \supset \cdots \supset SL_n(k) \supset \cdots
$$
and $\{SL_n(k)\}_{k\geqslant 1}$ is called the {\it Milnor filtration} of $SL_n$. Let $H_{\mathbb{Z}}$ denote the first homology group of $D_n$ with integer coefficients
$$
H_{\mathbb{Z}}:=H_1(D_n, \mathbb{Z})\cong \mathbb{Z}^{\oplus n}
$$
Then, for $k \geqslant 1$, the monoid homomorphism
\begin{equation}\label{eq:Milnor1}
\mu_k: SL_n(k) \longrightarrow  H_{\mathbb{Z}}\otimes \gr_{k}F_n \cong H_{\mathbb{Z}} \otimes \mathfrak{L}_k(H_{\mathbb{Z}})
\end{equation}
is defined by 
\begin{equation}\label{eq:exMilnor}
\mu_k(L)=\sum_{i=1}^n [x_i] \otimes y_i^{(k)}(L)\cong \sum_{i=1}^n X_i \otimes Y_i^{(k)}(L)
\end{equation}
and called the {\it Milnor invariant of degree $k$}. Here, we canonically identify $\gr_{k}F_n$ with the degree $k$-part $\mathfrak{L}_k(H_{\mathbb{Z}})$ of the graded free Lie algebra $\mathfrak{L}(H_{\mathbb{Z}})$ generated by $H_{\mathbb{Z}}$ and we denote by $Y_i^{(k)}(L)$ the corresponding element to $y_i^{(k)}(L)$. We note that the coefficients of $Y_i^{(k)}(L)$ is the first non-vanishing Milnor invariants of $L$ defined by Milnor (\cite{Mi}).

\section{Jacobi diagrams and a free nilpotent Lie algebra}
In this section, we recall the definition of Jacobi diagrams and its relation with a free Lie algebra. Then, we recall the 3rd homology group of a free nilpotent Lie algebra also has the diagrammatic interpretation. For more details on this subject, see \cite{MA}.
\subsection{Jacobi diagrams}\label{Jacobi}
Let $X(\neq \emptyset)$ be a compact oriented 1-manifold. A  {\it Jacobi diagram} is a uni-trivalent finite graph whose trivalent vertices are oriented, i.e., incident edges of each trivalent vertex are cyclically ordered. By convention, vertex orientations are given the trigonometric orientation of the plane. A {\it degree} of a Jacobi diagram is half the number of its univalent vertices and trivalent vertices. A {\it Jacobi diagram on X} is a Jacobi diagram whose univalent vertices are  attached to the interior of $X$ disjointly and each connected component of a Jacobi diagram has at least one univalent vertex in $X$ . We denote by $\mathcal{A}(X)$ the $\mathbb{K}$-vector space spanned by Jacobi diagrams on $X$ subject to AS, IHX and STU relations depicted as the following figures:
\ \\
\begin{center}
\begin{tikzpicture}
\begin{scope} 
\begin{scope}[densely dashed]
\draw (0,0)--(0,-0.45);
\draw (0,0)--(0.424,0.424);
\draw (0,0)--(-0.424,0.424);
\draw (-1,0) node{$=$};
\draw (-0.6,0) node{$-$};
\draw (-1,-0.9) node{AS};
\end{scope}
\begin{scope}[xshift=-1.7cm]
\draw[densely dashed] (0,0)--(0,-0.45);
\draw[densely dashed] (0,0) .. controls (0.3,0.1) and (0.3,0.25)..(-0.424,0.424);
\draw[densely dashed] (0,0) .. controls (-0.3,0.1) and (-0.3,0.25)..(0.424,0.424);
\end{scope}
\end{scope}
\begin{scope}[xshift=-0.5cm]
\begin{scope}[xshift=2cm,densely dashed]
\draw (-0.3, 0.4) --(0.3,0.4);
\draw (0,0.4)--(0,-0.4);
\draw (-0.3,-0.4) --(0.3,-0.4);
\end{scope}
\begin{scope}[xshift=3.5cm, densely dashed]
\draw (-0.3,0.4)--(-0.3,-0.4);
\draw (0.3,0.4)--(0.3,-0.4);
\draw (-0.3,0)--(0.3,0);
\draw (-0.8, 0) node{$-$};
\end{scope}
\begin{scope}[xshift=5cm, densely dashed]
\draw (-0.3,0.4)--(0.3,-0.4);
\draw (0.3,0.4)--(-0.3,-0.4);
\draw (-0.19,-0.26)--(0.19,-0.26);
\draw (-0.8, 0) node{$+$};
\draw (0.8,0) node{$=0$};
\draw (-1.5, -0.9) node{IHX};
\end{scope}
\end{scope}
\begin{scope}[yshift=-2cm]
\begin{scope}
\draw[densely dashed] (0,0)--(0,-0.45);
\draw[densely dashed] (0,0)--(0.424,0.424);
\draw[densely dashed] (0,0)--(-0.424,0.424);
\draw[thick,->] (0.7,-0.45)--(-0.7,-0.45);
\draw (1,0) node{$=$};
\end{scope}
\begin{scope}[xshift=2cm]
\draw[densely dashed] (0.424,-0.45)--(0.424,0.424);
\draw[densely dashed] (-0.424,-0.45)--(-0.424,0.424);
\draw[thick,->] (0.7,-0.45)--(-0.7,-0.45);
\draw (1,0) node{$-$};
\draw (0, -0.9) node{STU};
\end{scope}
\begin{scope}[xshift=4cm]
\draw[densely dashed] (0.424,-0.45)--(-0.424,0.424);
\draw[densely dashed] (-0.424,-0.45)--(0.424,0.424);
\draw[thick,->] (0.7,-0.45)--(-0.7,-0.45);
\end{scope}
\end{scope}
\end{tikzpicture}
\end{center}
Conventionally, the components of $X$ are depicted as solid line whereas the Jacobi diagrams are depicted as dashed one. The space $\mathcal{A}(X)$ is graded by the degree of Jacobi diagrams and we denote by same $\mathcal{A}(X)$ its degree completion , i.e., $\mathcal{A}(X)=\prod_{d\geqslant 0}\mathcal{A}_d(X)$ where $\mathcal{A}_d(X)$ is degree $d$ part of $\mathcal{A}(X)$. For finite set $S$, we also set $\mathcal{A}(\uparrow_S)=\mathcal{A}(\amalg_{s \in S} \uparrow_s)$ where $\uparrow_s$ denote the copy of (oriented) unit interval $[0,1]$ for each element $s \in S$. By staking, $\mathcal{A}(\uparrow_S)$ is endowed with structure of algebra. Moreover, for all $X$ the vector space $\mathcal{A}(X)$ has a structure of Hopf algebra by a natural cocommutative comultiplication $\Delta$ defined as in \cite{BN}. Hence, we can define primitive elements and  group-like elements as follows: An element $\xi \in \mathcal{A}(\uparrow S)$ is called a {\it primitive element} if  $\Delta(\xi)=\xi\otimes 1 +1\otimes \xi$ and an element $\xi \in \mathcal{A}(\uparrow_S)$ is called a {\it group-like element} if $\Delta(\xi)=\xi\otimes \xi$ and $\epsilon(\xi)=1$. Here, $\epsilon: \mathcal{A}(\uparrow_S) \rightarrow \mathbb{K}$ is given by $\epsilon(\xi)=1$ $($if $\xi$ has no dashed component$)$ and $\epsilon(\xi)=0$ $($otherwise$)$ for $\xi \in \mathcal{A}(\uparrow_S)$. 

Let $S (\neq \emptyset)$ be a finite set. Let $\mathcal{B}(S)$ denote the complete graded vector space spanned by Jacobi diagrams whose univalent vertices are labelled by elements of the set S and each connected component has at least one univalent vertex subject to AS and IHX relations. Here the completion is given by the degree of Jacobi diagrams. We note that $\mathcal{B}(S)$ has an algebraic structure given by disjoint union. There is a Poincar\'e-Birkhoff-Witt type isomorphism $\chi : \mathcal{B}(S) \rightarrow \mathcal{A}(\uparrow_S)$ which sends a Jacobi diagram to the average of all the ways of putting its univalent vertices labelled by $s$ to the interval $\uparrow_s$ for $s \in S$. Note that the isomorphism $\chi$ is not an algebra morphism. In particular, if $S$ is a finite dimensional vector space $H$, then $\mathcal{B}(H)$ is called the graded vector space of {\it $H$-colored} Jacobi diagrams. The vector space $\mathcal{B}(H)$ is also subject to multilinearity relation as the following picture:
\[
\begin{tikzpicture}
\begin{scope}[xshift=7cm, densely dashed]
\begin{scope} 
\draw (0,0)--(0,0.45);
\draw (0,0)--(-0.424, -0.424);
\draw (0,0) -- (0.424, -0.424);
\draw (0.8,0) node{$=$};
\draw (0, 0.7) node{$v_1 + v_2$};
\end{scope}
\begin{scope}[xshift=1.5cm, densely dashed]
\draw (0,0)--(0,0.45);
\draw (0,0)--(-0.424, -0.424);
\draw (0,0) -- (0.424, -0.424);
\draw (0,0.7) node{$v_1$};
\draw (0,-0.9) node{multilinearity};
\end{scope}
\begin{scope}[xshift=2.8cm, densely dashed]
\draw (0,0)--(0,0.45);
\draw (0,0)--(-0.424, -0.424);
\draw (0,0) -- (0.424, -0.424);
\draw (0,0.7) node{$v_2$};
\draw (-0.6, 0) node{$+$};
\end{scope}
\end{scope}
\end{tikzpicture}
\]
The element in $H$ assigned to a vertex $v$ of a diagram in $\mathcal{B}(H)$ is denoted by $\mathrm{col}(v)$.

Let $\mathcal{A}^t(\uparrow_S)$ be the graded quotient of $\mathcal{A}(\uparrow_S)$ by the subspace spanned by Jacobi diagrams containing a non-simply connected dashed component. Then, $\mathcal{B}^t(S):=\chi^{-1}(\mathcal{A}^t(\uparrow_S))$ is the commutative polynomial algebra of $\mathcal{C}^t(S)$. Here, we denote by $\mathcal{C}^t(S)$ the space of connected tree Jacobi diagrams labelled by elements of $S$.

Let $H$ be a vector space. We recall that, for any connected tree Jacobi diagram $A_r$ all of whose univalent vertices are $H$-colored except for $r$, we can assign the element $\comm(A_r)$ in the free lie algebra generated by $H$ as the following manner: 
For any univalent vertex $v \neq r$, label the edge incident to $v$ of $A_r$ by color $\col(v)$ of $v$. Next, we assign the label $[a,b]$ to any edge meeting $a$-labelled edge and $b$-labelled edge at a trivalent vertex following the cyclically orientation. Finally, one obtain the label associated the the edge incident to $r$. This is the desired element $\comm(A_r)$.
\begin{example}{\rm
For the following degree 2 $H$-colored Jacobi diagram , its associated element of Lie algebra is given by
\[
\begin{tikzpicture}[thick]
\begin{scope}[xshift=1cm]
\node at (-3,1) {$\mathrm{comm}\Biggl($};
\node at (3,1) {$\Biggr)=[[[v_1,v_2],[v_3,v_4]],v_5].$};
\end{scope}
\begin{scope}[scale=0.35]
\coordinate (v_1) at (-3,5);
\coordinate (v_2) at (-1,5);
\coordinate (v_3) at (0,4);
\coordinate (v_4) at (2,4);
\coordinate (v_6) at (5,5);
\coordinate (r_0) at (1,0);
\coordinate (r_1) at (-2,4);
\coordinate (r_2) at (1,1);
\coordinate (r_3) at (1,3);
\coordinate (r_4) at (4,4);
\coordinate (r_5) at (0,2);
\draw[densely dashed] (r_0)--(r_2)--(r_4)--(v_6);
\draw[densely dashed] (r_2)--(r_5)--(r_1)--(v_1);
\draw[densely dashed] (r_5)--(r_3)--(v_4);
\draw[densely dashed] (r_3)--(v_3);
\draw[densely dashed] (r_1)--(v_2);
\node[anchor=south] at (v_1) {$v_1$};
\node[anchor=south] at (v_2) {$v_2$};
\node[anchor=south] at (v_3) {$v_3$};
\node[anchor=south] at (v_4) {$v_4$};
\node[anchor=south] at (v_6) {$v_5$};
\node[anchor=north] at (r_0) {$r$};
\end{scope}
\end{tikzpicture}
\]
}
\end{example}

\subsection{The diagrammatic description of the 3rd homology group of a free nilpotent Lie algebra}
In this section, we recall the diagrammatic interpretation of the 3rd homology group of a free nilpotent Lie algebra in terms of the fission map introduced by Massuyeau in \cite{MA}.

Let $H$ donote a finite dimensional vector space. Let $\mathfrak{L}(H)$ be the free Lie algebra generated by $H$.  The free Lie algebra $\mathfrak{L}(H)$ has a grading induced by the commutator length, i.e., $\mathfrak{L}(H):=\bigoplus_{k\geqslant1}\mathfrak{L}_k(H)$.
For simplicity, we often denote  $\mathfrak{L}(H)$ by $\mathfrak{L}$. Since $\mathfrak{L}_{\geqslant k+1}$ coincides with $\Gamma_{k+1}\mathfrak{L}$, the Lie algebra $\mathfrak{L}/\mathfrak{L}_{k+1}$ is the free nilpotent Lie algebra generated by $H$ of nilpotency class $k$. The {\it Koszul complex} of $\mathfrak{L}$ with trivial coefficients is the chain complex $(\Lambda^{\bullet}( \mathfrak{L}/\mathfrak{L}_{k+1}), \partial)$ with  boundary operator $\partial_n: \Lambda^n \mathfrak{L} \rightarrow \Lambda^{n-1}\mathfrak{L}$ given by
$$
\partial_n (h_1\wedge \cdots \wedge h_n)=\sum_{i<j}(-1)^{i+j}\cdot [h_i,h_j]\wedge h_1\wedge \cdots\widehat{h}_i\cdots \widehat{h}_j \cdots \wedge h_n.
$$
Since the grading of $\mathfrak{L}$ induces a grading of the Koszul complex $\Lambda^{\bullet}(\mathfrak{L}/\mathfrak{L}_{\geqslant k+1})$, its homology groups are endowed with the structure of graded vector space. 

Next, we consider the central extension of graded Lie algebras
$$
0\rightarrow \mathfrak{L}_k \rightarrow \mathfrak{L}/\mathfrak{L}_{\geqslant k+1} \rightarrow \mathfrak{L}/\mathfrak{L}_{\geqslant k} \rightarrow1
$$
Let $\{ E_{p,q}^r\}$ be the Hochschild-Serre spectral sequence associated to the above central extension which gives
\begin{equation}\label{spectral}
E_{p,q}^r \overset{r \rightarrow \infty}{\longrightarrow}H_{p+q}(\mathfrak{L}/\mathfrak{L}_{\geqslant k+1}) \quad \mathrm{and} \quad E_{p,q}^2 \simeq H_p(\mathfrak{L}/\mathfrak{L}_{\geqslant k})\otimes \Lambda^q\mathfrak{L}_k.
\end{equation}
Let $D_k(H)$ be the kernel of the Lie bracket $[-,-]: H\otimes \mathfrak{L}_{k}(H) \rightarrow \mathfrak{L}_{k+1}(H)$ given by $X\otimes Z \mapsto [X,Z]$ for $X \in H$, $Z \in \mathfrak{L}_{k}$ . Then, we have the following theorem given by Igusa and Orr.

\begin{theorem}\label{thm:IgusaOrr}$($\cite{IO}$).$
There is an isomorphism of graded vector spaces
$$
\mathrm{IO}:H_3(\mathfrak{L}/\mathfrak{L}_{\geqslant k}) \overset{\simeq}{\longrightarrow}\bigoplus_{l=k}^{2k-2}D_l(H)
$$
such that we have the following commutative  diagram
$$
\xymatrix{
&H_{3}(\mathfrak{L}/\mathfrak{L}_{\geqslant m}) \ar[d] \ar[r]^{\simeq}&\bigoplus_{l=m}^{2m-2}D_l(H) \ar[d]\\
&H_{3}(\mathfrak{L}/\mathfrak{L}_{\geqslant k}) \ar[r]^{\simeq} &\bigoplus_{l=k}^{2k-2}D_l(H)}
$$
for all $m\geqslant k$ and such that the composition of the canonical projection and the isomorphism $\mathrm{IO}$
$$
H_3(\mathfrak{L}/\mathfrak{L}_{\geqslant k}) \xtwoheadrightarrow[]{} H_3(\mathfrak{L}/\mathfrak{L}_{\geqslant k})_{k+1} \xrightarrow[\simeq]{\mathrm{IO}_{k+1}} D_{k}(H) \subset H\otimes \mathfrak{L}_k
$$
coincides with the differencial $d_{3,0}^{2}:E_{3,0}^{2} \rightarrow E_{1,1}^{2}$ of the spectral sequence (\ref{spectral}). Here $H_3(\mathfrak{L}/\mathfrak{L}_{\geqslant k})_{k+1}$ denotes the degree $k+1$ part of $H_3(\mathfrak{L}/\mathfrak{L}_{\geqslant k+1})$.
\end{theorem}
Finally, we recall the diagrammatic description of $H_3(\mathfrak{L}/\mathfrak{L}_{\geqslant k})$. For this purpose, let us recall  the space $D_k(H)$ has a description in terms of tree Jacobi diagram. Let $T$ be a connected tree $H$-colored Jacobi diagram in $\mathcal{C}^{t}_k(H)$. As explained in section \ref{Jacobi},  for each univalent vertex $v_0$ with color $\mathrm{col}(v_0)$, we can assign the element $\mathrm{comm}(T_{v_0})$ in $\mathfrak{L}_{k}$ ,where $T_{v_0}$ denotes the tree diagram rooted at $v_0$. By taking a sum over all univalent vertices $v$ of $T$, we have a linear map
$$
\eta_k(T):=\sum_v\mathrm{col}(v) \otimes \mathrm{comm}(T_v).
$$
This linear map $\eta_k$ gives an isomorphism
$$
\eta_k:\mathcal{C}^t_k(H) \overset{\simeq}{\longrightarrow} D_{k}(H) \subset H\otimes \mathfrak{L}_{k}(H).
$$

For the diagrammatic discription of $H_3(\mathfrak{L}/\mathfrak{L}_{\geqslant k})$, we need to define the {\it fission map} of tree diagrams given in the following manner. Let $T$ be a degree $k$ connected tree $H$-colored Jacobi diagram. For each trivalent vertex $r$, we may consider $T$ as the union of three tree diagrams rooted at $r$, which we denote by $T_r^{(1)}, T_r^{(2)}$ and $T_r^{(3)}$. Here the numbering $1, 2, 3$ is defined according to  the cyclically ordering of $r$. By the method in \S2.1, we can associate the element $\mathrm{comm}(T_r^{(i)})$ of $\mathfrak{L}$ to each tree $T_r^{(i)}$ for $i=1,2,3$. Then, the fission map $\phi : \mathcal{C}^t(H) \rightarrow \Lambda^3 \mathfrak{L}$ is defined by
$$
\phi(T):=\sum_r \mathrm{comm}(T_r^{(3)}) \wedge \mathrm{comm}(T_r^{(2)}) \wedge \mathrm{comm}(T_r^{(1)})
$$
where the sum is taken over all trivalent vertices $r$ of $T$. Moreover, it turns out that its image under the boundary operator $\partial_3$ is given by
\begin{equation}
\partial_3(\phi(T))=\sum_{v} \col (v) \wedge \comm(T_v) \in \Lambda^2 \mathfrak{L}
\end{equation}
where the sum is taken over all univalent vertices $v$ of $T$.

By the fission map and theorem \ref{thm:IgusaOrr}, we have the following thereom which gives the description of $H_3(\mathfrak{L}/\mathfrak{L}_{\geqslant k})$ in terms of tree diagrams. For the proof, see \cite{MA}

\begin{theorem}\label{thm:Phi}$($\cite{MA}$)$
The fission of tree diagrams defines a linear isomorphism
$$
\Phi: \bigoplus_{l=k}^{2k-2}\mathcal{C}_l^t(H) \overset{\simeq}{\longrightarrow}H_3(\mathfrak{L}/\mathfrak{L}_{\geqslant k})
$$
which shifts the degree by $+1$.\footnote{Here, if we define a degree as the number of trivalent vertices, then $\Phi$ shifts the degree +2 as in \cite{MA}}
\end{theorem}


\section{Special expansions}
In this section, we review the Malcev completion and Malcev Lie algebra of a group. Then, we recall the group-like expansion  and special expansions intoroduced by Massuyeau in \cite{MA} and \cite{MA2}. For more information on the Malcev completion and the Malcev Lie algebra and group-like expansion, consult \cite{Q} and \cite{MA} respectively.
\subsection{The Malcev completion and the Malcev Lie algebra of a group}

Let $G$ be a group. Let $\mathbb{K}[G]$ be the group ring of $G$ over $\mathbb{K}$ and $\epsilon :\mathbb{K}[G] \rightarrow \mathbb{K}$ be the augmentation map. Let $I:=\Ker \epsilon$ be the augmentation ideal of $\mathbb{K}[G]$. The augmentation ideal $I$ is generated by $(g-1)$ $(g \in G)$ as a $\mathbb{K}$-module. We define the $I$-adic completion of $\mathbb{K}[G]$ by
$$
\widehat{\mathbb{K}[G]}:=\varprojlim_{k}\mathbb{K}[G]/I^k
$$
The $I$-adic completion $\widehat{\mathbb{K}[G]}$ is equipped with the filtaration given by $\widehat{I_{j}}:=\varprojlim_{k\geqslant j}I^j/I^k$  for all $j \geqslant 0.$
We define the coproduct $\Delta : \mathbb{K}[G] \rightarrow \mathbb{K}[G] \otimes \mathbb{K}[G]$ by $\Delta(g)=g \otimes g$ $(g \in G)$. Since the coproduct $\Delta$ is continuous with respect to the $I$-adic topology, $\Delta$ induces the coproduct $\widehat{\Delta}: \widehat{\mathbb{K}[G]} \rightarrow \widehat{\mathbb{K}[G]} \widehat{\otimes} \widehat{\mathbb{K}[G]}$ on $\widehat{\mathbb{K}[G]}$.
Hence, the $I$-adic completion $\widehat{\mathbb{K}[G]}$ is endowed with a structure of complete Hopf algebra.

The {\it Malcev completion} $\mathsf{M}(G)$ of $G$ is the subgroup of $\widehat{\mathbb{K}[G]}$ consisting of group-like elements, i.e.,
$$
\mathsf{M}(G):=\left\{x \in \widehat{\mathbb{K}[G]} \mid \widehat{\Delta}(x)=x \widehat{\otimes} x, \epsilon(x)=1\right\}.
$$
By the filtration of $\widehat{\mathbb{K}[G]}$, $\mathsf{M}(G)$ is quipped with the filtration $\widehat{\Gamma}_j \mathsf{M}(G):=\mathsf{M}(G) \cap \left( 1 + \widehat{I^j}\right)$ for all $j \geqslant 1.$
The {\it Malcev Lie algebra} of  $G$ is the Lie algebra of primitive elements of $\widehat{\mathbb{K}[G]}$, i.e.,
$$
\mathfrak{m}(G):=\left\{ x \in \widehat{\mathbb{K}[G]} \mid \widehat{\Delta}(x)=x\widehat{\otimes}1 +1\widehat{\otimes}x \right\}
$$
with the induced filtration $\widehat{\Gamma_j}\mathfrak{m}(G):=\mathfrak{m}(G) \cap \widehat{I^j}$ for all $j \geqslant 1.$

Since exponential $\exp$ and logarithmic series $\log$ give the one-to-one correspondence between the primitive and group-like elements in complete Hopf algebra, the Malcev completion $\mathsf{M}(G)$ and the Malcev Lie algebra $\mathfrak{m}(G)$ of $G$ are equivalent, where $\exp$ and $\log$ is given by
$$
\exp(x):=\sum_{m=0}^{\infty}\frac{x^m}{m!}, \quad \forall x \in \mathfrak{m}(G), \quad \log(y):=\sum_{m=1}^{\infty}(-1)^{m-1} \frac{(y-1)^m}{m}, \quad \forall y \in \mathsf{M}(G).
$$

\subsection{ Group-like expansion}
Let $F_n$ be a free group and we choose a generating system $\mathbf{x}=\{x_1,\ldots, x_n\}$ of $F_n$. Let $H$ denote the abelianization of $F_n$ with $\mathbb{K}$ coefficients:
$$
H:=F_n/\Gamma_2 F_n\otimes \mathbb{K}.
$$
We set $X_i:=[x_i]\otimes_{\mathbb{Z}}1 \in H$ for $1\leqslant i\leqslant n$.
Let $T(H)$ be the tensor algebra generated by $H$ and $\widehat{T}(H)$ be its degree completion, i.e., $T(H)=\bigoplus_{k\geqslant 0}H^{\otimes k}$,  $\widehat{T}(H)=\prod_{k\geqslant 0}H^{\otimes k}.$
We note that $\widehat{T}(H)$ may be identified with $\mathbb{K}\langle \langle X_1,\ldots,X_n \rangle \rangle$.

A monoid map $\theta : F_n \rightarrow \widehat{T}(H)$ is called a {\it $\mathbb{K}$-valued Magnus expansion} of the free group $F_n$ if $\theta(x)\equiv 1+\{x\}\ \mod \ \prod_{k \geqslant 2}H^{\otimes k}$ for each $x\in F_n$. We restrict our selves to a special type of $\mathbb{K}$-valued Magnus expansions as follows. A $\mathbb{K}$-valued Magnus expansion $\theta: F_n \rightarrow \widehat{T}(H)$ is called a {\it group-like expansion} if it takes values in the group of group-like elements of $\widehat{T}(H)$, i.e., $\Delta(\theta(x))=\theta(x)\otimes \theta(x)$ and $\epsilon(\theta(x))=1$.

\begin{remark}\label{rem:group}{\rm (Properties of group-like expansions)\\
(1)\ A group-like expansion $\theta$ of $F_n$ extends to a unique complete Hopf algebra isomorphism
$$
\theta: \widehat{\mathbb{K}[F_n]} \overset{\simeq}{\longrightarrow}\widehat{T}(H)
$$
which is the identity at the graded level. Conversely, any such isomorphism $\theta$ restricts to a group-like expansion $\theta: F_n \rightarrow \widehat{T}(H)$.\\
(2)\ A group-like expansion $\theta$ of $F_n$ induces a unique filtered Lie algebra isomorphism
$$
\theta: \mathfrak{m}(F_n) \overset{\simeq}{\longrightarrow} \widehat{\mathfrak{L}}(H)
$$
which is the identity at the graded level. Conversely, any such isomorphism $\theta$ induces a unique group-like expansion $\theta: F_n \rightarrow \widehat{T}(H)$.\\
(3)\ Let $R$ be a normal subgroup of $F_n$ and let $\theta$ be a group-like expansion of $F_n$. We denote by $\langle \langle \log \theta(R)\rangle \rangle$ the closed ideal of $\widehat{\mathfrak{L}}(H)$ generated by $\log \theta(R)$. Then, $\theta : \mathfrak{m}(F_n) \rightarrow \widehat{\mathfrak{L}}(H)$ induces a unique filtered Lie algebra isomorphism
$$
\theta: \mathfrak{m}(F_n/R) \overset{\simeq}{\longrightarrow} \widehat{\mathfrak{L}}(H)/ \langle\langle \log \theta(R) \rangle \rangle.
$$}
\end{remark}

By Remark \ref{rem:group}(1), for any two expansions $\theta$ and $\theta'$ of $F_n$, there exists a unique filtered algebra automorphism $\psi: \widehat{T}(H) \rightarrow \widehat{T}(H)$ inducing the identity at the graded level and such that $\psi \circ \theta =\theta'$.


\subsection{Special expansions}
Let $D_n$ be a $n$ punctured disk. Let $\pi_1(D_n)$ be the fundamental group of $D_n$. Then, $\pi_1(D_n)$ can be identified with the free group $F_n$ of rank $n$ on $x_1,\ldots, x_n$ where $x_i$ represents a small loop around $i$-th puncture $(1\leqslant i \leqslant n)$. We note that $x_1\cdots x_n$ represents the boundary curve of $D_n$.

Let $H$ denote the first homology group of $D_n$ with $\mathbb{K}$ coefficients:
$$
H:=H_1(D_n, \mathbb{Z})\otimes \mathbb{K}
$$
The homology group $H$ is generated by $X_i:=[x_i]\otimes_{\mathbb{Z}}1$ $(1\leqslant i\leqslant n)$ corresponding to the generator $x_i$ $(1\leqslant i\leqslant n)$ of $\pi_1(D_n, p)$.

\begin{definition}$($\cite{MA2}$)${\rm
A group-like expansion  $\theta: F_n \rightarrow \widehat{T}(H)$ is called a {\it special expansion} if $\theta$ satisfies the following two conditions:\\
(1) (tangential condition) $\theta(x_i)=U_i\exp(X_i)U_i^{-1}$ for some $U_i \in \exp(\mathfrak{L})$\\
(2) (normalized condition) $\theta(x_1\cdots x_n)=\exp(X_1+\cdots +X_n)$
}\end{definition}


\section{Total Milnor invariants for the special derivations}
In this section, we introduce the total Milnor invariant for the special derivations of the complete graded free Lie algebra. Then, we define the total Milnor invariant for the monoid of string links in terms of a special expansion. Finally, we show some properties of it.
\subsection{The automorphism group of complete graded free Lie algebra}
Here, we recall the automorphism group of complete graded free Lie algebra following \cite{AT}. Let $\mathfrak{L}(H)$ be the graded free Lie algebra generated by $H$. We also denote by $\mathfrak{L}(H)$ its degree completion. We often simply denote by $\mathfrak{L}$ abbreviating $H$. We note that the universal enveloping algebra $\mathcal{U}\mathfrak{L}$ is the ring $\mathbb{K}\langle \langle X_1,\ldots, X_n \rangle \rangle$ 
of noncommutative power series over $\mathbb{K}$ on $n$ indeterminates  $X_1,\ldots, X_n$.

Let $\Der(\mathfrak{L})$ be the Lie algebra of derivations of $\mathfrak{L}$. Note that  each derivation $\psi \in \Der(\mathfrak{L})$ is completely determined by its values $\psi(X_1), \ldots, \psi(X_n) \in \mathfrak{L}$. 

Hence, the grading of $\Der(\mathfrak{L})$ is induced by the grading of $\mathfrak{L}$. 

For each $\psi \in \Der(\mathfrak{L})$, $\psi$ is called a {\it tangential derivation} of $\mathfrak{L}$ if there exists $Y_i \in \mathfrak{L}$ such that $\psi(X_i)=[Y_i, X_i]$ $(1\leqslant i \leqslant n)$. We denote by $\mathrm{tDer}(\mathfrak{L})$ the subspace  generated by tangential derivations and we may see that $\mathrm{tDer}(\mathfrak{L})$ forms a Lie subalgebra of $\Der(\mathfrak{L})$ (cf. \cite[Proposition 3.4]{AT}).

For each $\psi \in \mathrm{tDer}(\mathfrak{L})$, $\psi$ is called a special derivation if $\psi(X_1+\cdots+X_n)=0$, i.e. $[Y_1,X_1]+\cdots+[Y_n,X_n]=0$. Special derivations form a Lie subalgebra of $\mathfrak{L}$ and we denote it by $\mathrm{sDer}(\mathfrak{L})$

Since each $\psi \in \mathrm{sDer}(\mathfrak{L})$ is completely determined by its values $\psi(X_i)=[Y_i,X_i]$, we may identify $\mathrm{sDer}(\mathfrak{L})$ with $\bigoplus_{i=1}^n\mathfrak{L}/\langle X_i \rangle$ by the correspondence $\psi\mapsto (Y_i,\ldots, Y_n)$. In the following, for each $\psi \in \mathrm{sDer}(\mathfrak{L})$, we often write that  $\psi=(Y_1,\ldots, Y_n)$ by this identification.

Let $\mathrm{sAut}(\mathfrak{L})$ be the group of automorphisms which sends $X_i$ to $U_i X_i U_i^{-1}$ $(1\leqslant i \leqslant n)$ for some $U_i \in \exp(\mathfrak{L})$ and fix $X_1+\cdots +X_n$, i.e.
$$
\mathrm{sAut}(\mathfrak{L})=\left\{ \varphi \in \Aut(\mathfrak{L}) \left|
\begin{array}{l}
 \varphi(X_i)=U_i X_i U_i^{-1}\ (1\leqslant i \leqslant n)\\
 \mathrm{for\ some\ } U_i \in \exp(\mathfrak{L}) \\
\varphi(X_1+\cdots +X_n)=X_1+\cdots +X_n\\
 \end{array}\right.
\right\}.
$$
For each $\varphi \in \mathrm{sAut}(\mathfrak{L})$, we also write that $\varphi=(U_1,\ldots, U_n)$ as above under the condition that the coefficient of $X_i$ in $\log (U_i)$ is 0.  Then, the exponential $\exp$ and the logarithm $\log$ gives the bijection between $\mathrm{sDer}(\mathfrak{L})$ and $\mathrm{sAut}(\mathfrak{L})$.

\subsection{The total Milnor invariant for special derivation}
Let us define the total Milnor invariant for special derivations. 
\begin{definition}
The map 
$$
\mu: \mathrm{sDer}(\mathfrak{L}) \longrightarrow H \otimes \mathfrak{L}
$$
given by 
$$
\mu(\psi)=\sum_{i=1}^n X_i\otimes Y_i
$$
for each $\psi=(Y_1,\ldots, Y_n) \in \mathrm{sDer}(\mathfrak{L})$ is called the {\it total Milnor invariant }of $\psi$.
\end{definition}
\begin{remark}{\rm 
By composing $\mu$ with $\log$, we have the map
$$
\log\circ \mu: \mathrm{sAut}(\mathfrak{L}) \longrightarrow H \otimes \mathfrak{L}
$$
given by 
$$
\mu(\log(\varphi))=\sum_{i=1}^n X_i\otimes Y_i
$$
for each $\varphi=(\exp Y_1, \ldots, \exp Y_n) \in \mathrm{sAut}(\mathfrak{L})$. 
}
\end{remark}
In the following, we often denote  $\log\circ \mu$ simply by $\mu$ as far as no risk of confusion. Let us consider the bracket map $[-,-]: H\otimes \mathfrak{L} \rightarrow \mathfrak{L}; X\otimes Y \mapsto [X,Y]$.
\begin{lemma}\label{D}
The image of the total Milnor invariant $\mu(\mathrm{sDer}(\mathfrak{L}))$ is contained in $D(H):=\mathrm{Ker}([-,-])$.
\end{lemma}
\begin{proof}
For each special derivation $\psi=(Y_1,\dots, Y_n)$, we have $[X_1,Y_1]+\cdots +[X_n,Y_n]=0$ by speciality. Hence, the assertion immediately follows.
\end{proof}
\subsection{The total Milnor invariant for string links}

Here, we give the ``infinitesimal'' version of the Artin representation of pure braid group. For $\phi \in \Aut(F_n)$, $\psi$ induces the automorphism $\widehat{\phi}$ of the complete Hopf algebra $\widehat{\mathbb{K}[F_n]}$. Hence, by restricting to the primitive part, we obtain a filtered Lie algebra isomorphism $\mathfrak{m}(\phi): \mathfrak{m}(F_n) \rightarrow \mathfrak{m}(F_n)$. Then, we have a group homomorphism
$$
\mathfrak{m}: \Aut(F_n) \longrightarrow \Aut^{\mathrm{fil}}(\mathfrak{m}(F_n)), \quad \phi \mapsto \mathfrak{m}(\phi).
$$
where $\Aut^{\mathrm{fil}}(\mathfrak{m}(F_n))$ denotes the group of filtration preserving automorphisms of $\mathfrak{m}(F_n)$. Moreover, it turns out that we have an isomorphism
$$
\mathfrak{m}: \Aut(F_n) \overset{\sim}{\longrightarrow} \Aut^{\mathrm{fil}}_{\log(F_n)}(\mathfrak{m}(F_n)), \quad \phi \mapsto \mathfrak{m}(\phi).
$$
where we set $\Aut(G)^{\mathrm{fil}}_{\log(G)}(\mathfrak{m}(G)):= \{ \psi \in \Aut(\mathfrak{m}(G)) \mid \psi(\log(G))=\log(G)\}$. Hence, by composing $Art$ in \S1.1 with $\mathfrak{m}$, we obtain the infinitesimal Artin representation
$$
\mathfrak{m}(Art):P_n \longrightarrow \Aut_0(\mathfrak{m}(F_n)), \quad L\mapsto \mathfrak{m}(Art(L))
$$
where $\Aut_0(\mathfrak{m}(F_n))$ denotes the following subgroup of $\Aut(\mathfrak{m}(F_n))$:
$$
\Aut_0(\mathfrak{m}(F_n)):=\left\{\Phi \in \Aut(\mathfrak{m}(F_n)) \left|
\begin{array}{l}
 \Phi(\log(x_i))=\log(y_i x_iy_i^{-1})\\
 \mathrm{for\ some\ } y_i \in F_n \ (1\leqslant i \leqslant n) \\
 \Phi(\log(x_1\cdots x_n))=\log(x_1\cdots x_n)\\
 \Phi(\log(F_n))=\log(F_n)
 \end{array}\right.
\right\}.
$$

Let $\theta: \mathfrak{m}(F_n) \rightarrow \widehat{\mathfrak{L}}$ be a special expansion. 
For any $\phi \in \Aut(\mathfrak{m}(F_n))$, we define the automorphism $\theta^{\ast}(\phi):=\theta \circ \phi \circ \theta^{-1} \in \Aut^{\mathrm{fil}}(\widehat{\mathfrak{L}})$, where $\Aut(\widehat{\mathfrak{L}})$ denotes the subgroup of $\Aut^{\mathrm{fil}}(\widehat{\mathfrak{L}})$ consisting of filtration-preserving automorphisms. Hence, we have a homomorohism, called the {\it special Artin representation}, 
$$
Art^{\theta}: P_n \rightarrow \Aut_0(F_n) \rightarrow \Aut_0(\mathfrak{m}(F_n)) \overset{\theta^{\ast}}{\rightarrow}\mathrm{sAut}(\mathfrak{L})
$$

Then, we have the following proposition.
\begin{proposition}
The image of $Art^{\theta}$ is given by
$$
Art^{\theta}(PB_n)=\left\{\varphi \in \Aut(\mathfrak{L}) \left|
\begin{array}{l}
\varphi(X_i)=(U_i^{-1}\theta(y_i)U_i)X_i (U_i^{-1}\theta(y_i)U_i)^{-1}\\
 (1\leqslant i \leqslant n) \\
\varphi(X_1+\cdots +X_n)
 \end{array}\right.
\right\}
$$
when $\theta(x_i)=U_i \exp(X_i) U_i^{-1}$ for some $U_i \in \exp(\mathfrak{L})$ $(1\leqslant i \leqslant n)$.
\end{proposition}
\begin{proof}
Since $\theta$ is a special expansion, we have $\theta^{-1}(X_i)=\log(g_ix_ig_i^{-1})$ for $g_i \in \exp(\mathfrak{m}(F_n))$ such that $\theta(g_i)=U_i^{-1}$ $(1\leqslant i \leqslant n)$. Then, we have
\begin{eqnarray*}
&&\theta\circ \mathfrak{m}(Art(L))\circ \theta^{-1}(X_i)\\
&=&\theta\circ \mathfrak{m}(Art(L))(\log(g_ix_ig_i^{-1}))\\
&=&\theta(\log(\mathfrak{m}(Art)(L)(g_ix_ig_i^{-1}))\\
&=&\theta(\log(g_ix_ig_i^{-1}|_{x_j=y_jx_jy_j^{-1}})\\
&=&\log(\theta(g_ix_ig_i^{-1})|_{\theta(x_j)=\theta(y_jx_jy_j^{-1})})\\
&=&{X_i}|_{U_jX_jU_j^{-1}=\theta(y_j)U_jX_jU_j^{-1}\theta(y_j)^{-1}}\\
&=&U_i^{-1}\theta(y_i)U_iX_iU_i^{-1}\theta(y_i)^{-1}U_i.
\end{eqnarray*}
Hence, the assertion follows.
\end{proof}

Then, we may define the total Milnor invariant for the pure braid group $PB_n$ as follows: For $L \in PB_n$, we set $\mu^{\theta}=\mu \circ \log \circ Art^{\theta}$, i.e. we have
$$
\mu^{\theta}: PB_n \longrightarrow H\otimes \mathfrak{L};\quad  L \mapsto \mu(\log(Art^{\theta}(L))).
$$
Then, we have the following theorem.
\begin{theorem}\label{Milnor}
The restriction of the degree $k$ part of the total Milnor invariant to the $k$-th term of the Milnor filtration $PB_n(k):=PB_n\cap SL_n(k)$ coincides with the Milnor invariant of degree $k$, i.e. we have
$$
\mu_k^{\theta}|_{PB_n(k)}=\mu_k \in \Hom(PB_n(k), H\otimes \mathfrak{L}_{k}).
$$
\end{theorem}
\begin{proof}
For $L \in PB_n(k)$, we have $Art^{\theta}(L)(X_i)=U_i^{-1}\theta(y_i(L))U_i$ and $\log\theta(y_i(L))=Y_i^{(k)}(L)+(\text{higer degree terms})$. Since $\theta$ is a Magnus expansion, $Y_i^{(k)}(L)$ is independent of choice of $\theta$, especially, we have
$$
\theta^M(y_i)=1+Y_i^{(k)}(L)+(\text{higer degree terms})
$$
where $\theta^M$ is the standard Magnus expansion defined by $x_i\mapsto 1+X_i$. Moreover, by direct computation, we may see that $\log(U_i^{-1}\theta(y_i(L))U_i)=Y_i^{(k)}(L)+(\text{higher degree terms})$. Hence, by definition of $\mu^{\theta}$, we conclude that $\mu_k^{\theta}|_{PB_n(k)}=\mu_k$. This completes the proof.
\end{proof}

As in \S1.1 the Artin representation of the pure braid group can be extended to the monoid of string links. Hence, we also have the infinitesimal Artin representation of the monoid of string links:
$$
\mathfrak{m}(Art_k(L)): SL_n \longrightarrow \Aut_0(\mathfrak{m}(F_n/\Gamma_{k+1}F_n)); \quad L\mapsto \mathfrak{m}(Art_k(L)).
$$
One may see that $\Ker(\mathfrak{m}(Art_k(L)))=SL_n(k)$. In terms of a special expansion $\theta$, we have the special Artin representation
$$
Art_k^{\theta}: SL_n \longrightarrow \mathrm{sAut}(\mathfrak{L}/\mathfrak{L}_{\geqslant k+1}); \quad L\mapsto \theta^{\ast}( \mathfrak{m}(Art_k(L))).
$$
Hence, we have the truncated total Milnor invariant
$$
\mu_{[1,k[}^{\theta}: SL_n \longrightarrow H\otimes \mathfrak{L}_{\geqslant 1}/\mathfrak{L}_{\geqslant k};\quad L \mapsto \mu(\log(Art^{\theta}_k(L))).
$$
By taking the inverse limit of the maps $\mu_{[1,k[}^{\theta}$ with respect to $k$, we obtain a map
$$
\mu^{\theta}:SL_n \longrightarrow H \otimes \mathfrak{L}
$$
since we have the projection $\mu^{\theta}_{[1,k[} \rightarrow \mu^{\theta}_{[1,l[}$ induced by the canonical projection $H\otimes \mathfrak{L}_{\geqslant 1}/\mathfrak{L}_{\geqslant k} \rightarrow H\otimes \mathfrak{L}_{\geqslant 1}/\mathfrak{L}_{\geqslant l}$ for $1\leqslant k \leqslant l$. The map $\mu^{\theta}$ restricts to the total Milnor invariant on $P_n$. And we can easily extend  Theorem \ref{Milnor} to the case of $SL_n$.
\begin{theorem}
The restriction of the degree $k$ part of the total Milnor invariant to the $k$-th term of the Milnor filtration $SL_n(k)$ coincides with the Milnor invariant of degree $k$, i.e. we have
$$
\mu_k^{\theta}|_{SL_n(k)}=\mu_k \in \Hom(SL_n(k), H\otimes \mathfrak{L}_{k}).
$$

\end{theorem}
\subsection{The diagrammatic interpretation of $\mu^{\theta}$}
In this section, we show that the truncation of $\mu^{\theta}$ is a monoid homomorphism. In the following, we keep the notation as in \S 6.3.
\begin{proposition}\label{truncation}
The restriction of degree $[k, 2k[$ truncation of $\mu^{\theta}$ to the $k$-th Milnor submonoid $SL_n$
$$
\mu^{\theta}_{[k,2k[}:=\sum_{m=k}^{2k-1}\mu_m^{\theta}: SL_n(k) \longrightarrow \bigoplus_{m=k}^{2k-1}H\otimes \mathfrak{L}_m
$$
is a monoid homomorphism and its kernel is $SL_n(2k)$.
\end{proposition}
\begin{proof}
By definition of special  Artin representation $Art^{\theta}$, for $L, L' \in SL_n(k)$, we have
\begin{eqnarray*}
&Art^{\theta}(L\circ L')(X_i)&=Art^{\theta}(L)(Art^{\theta}(L')(X_i))\\
&&=Art^{\theta}(L)(\exp(Y_i(L'))X_i\exp(Y_i(L'))^{-1}\\
&&=\exp(Art^{\theta}(L)(Y_i(L')))\exp(Y_i(L))X_i(\exp(Art^{\theta}(L)(Y_i(L')))\exp(Y_i(L)))^{-1}.
\end{eqnarray*}
By Balker-Campbell-Haussdorff formula, we have
\begin{eqnarray*}
&\log(\exp(Art^{\theta}(L)(Y_i(L')))\exp(Y_i(L)))&\equiv Art^{\theta}(L)(Y_i(L'))_{[k,2k[}+Y_i(L)_{[k,2k[} \mod \mathfrak{L}_{\geqslant 2k}\\
&&\equiv Y_i(L)_{[k,2k[}+Y_i(L')_{[k,2k[} \mod \mathfrak{L}_{\geqslant 2k}.
\end{eqnarray*}
Here, the last equality follows from the fact that $Art^{\theta}(L)$ acts on $\mathfrak{L}_{\geqslant k}/\mathfrak{L}_{2k}$ trivially and we denote by  $Y_i(L)_{[k,2k[}$ and  $Y_i(L')_{[k,2k[}$ the degree $[k,2k[$ truncation of $Y_i(L)$ and $Y_i(L')$. Hence, by definition of $\mu^{\theta}$, we have
\begin{eqnarray*}
&\mu_{[k,2k[}^{\theta}(L\circ L)&=\sum_{i=1}^{n}X_i\otimes \log(\exp(Art^{\theta}(L)(Y_i(L')))\exp(Y_i(L))) \mod \mathfrak{L}_{\geqslant 2k}\\
&&=\sum_{i=1}^n X_i \otimes (Y_i(L)_{[k,2k[}+Y_i(L')_{[k,2k[})\\
&&=\sum_{i=1}^n X_i \otimes Y_i(L)_{[k,2k[}+\sum_{i=1}^n X_i\otimes Y_i(L')_{[k,2k[}\\
&&=\mu_{[k,2k[}^{\theta}(L)+\mu_{k,2k[}^{\theta}(L').
\end{eqnarray*}
The kernel of $\mu_{[k,2k[}^{\theta}$ is clearly $SL_n(2k)$ from the definition of $\mu^{\theta}$. This completes the proof.
\end{proof}
From this proposition, we can deduce that the following algebraic structure of $SL_n$.
\begin{corollary}
The quotient monoid $SL_n(k)/SL_n(2k)$ is torsion-free abelian monoid.
\end{corollary}
By composing this proposition and Lemma \ref{D}, we have the following theorem.
\begin{theorem}
Notations being as above, for each $L \in SL_n(k)$, the degree $[k,2k[$ truncation $\mu^{\theta}_{[k,2k[}$ of $\mu^{\theta}$ is a monoid homomorphism and takes values in $\bigoplus_{m=k}^{2k-1} D_m(H)$. Hence, we have the following diagram valued monoid homomorphism
$$
\eta_{[k,2k[}^{-1}\circ \mu_{[k,2k[}^{\theta}=\bigoplus_{m=k}^{2k-1}\eta_m^{-1}\circ \mu_m^{\theta}: SL_n(k)\longrightarrow \bigoplus_{m=k}^{2k-1}\mathcal{C}_{m}^t(H).
$$

\end{theorem}
\section{The Infinitesimal Morita-Milnor homomorohism}
Morita introduced a refinement of the $k$-th Johnson homomorphism, called the {\it Morita homomorphism},  as a homomorphism from the $k$-th Johnson subgroup of the mapping class group of a surface to the 3rd homology group $H_3(F_n/\Gamma_{k+1}F_n)$ in \cite{M2}. And then, Massuyeau constructed its infinitesimal version, called the {\it infinitesimal Morita homomorphism}, as a homomorphism from the $k$-th Johnson subgroup to the 3rd homology group $H_3(\mathfrak{m}(F_n/\Gamma_{k+1}F_n))$ in \cite{MA}. In this section, we define the infinitesimal Morita-Milnor homomorphism which is a string analogue of infinitesimal Morita homomorphism.

\subsection{The infinitesimal Morita-Milnor homomorphism}

For each integer $k \geqslant 0$, we want to define a map
$$
M_{k+1}^{\theta}: SL_n(k+1) \longrightarrow H_3(\mathfrak{m}(F_n/\Gamma_{k+1}F_n)).
$$
For this purpose, we need the following  lemma: For any $l >m$, the canonical projection $t_{l.m}:\mathfrak{L}/\mathfrak{L}_{\geqslant l} \rightarrow \mathfrak{L}/\mathfrak{L}_{\geqslant m}$ induces the liner map ${t_{l,m}}_{\ast}: H_{\ast}(\mathfrak{L}/\mathfrak{L}_{\geqslant l}) \rightarrow H_{\ast}( \mathfrak{L}/\mathfrak{L}_{\geqslant m})$. Then we have
\begin{lemma}\label{lem:pre}$($\cite[Lemma 4.1]{MA}$)$ The linear map
$$
{t_{l,m}}_{\ast}: H_{2}( \mathfrak{L}/\mathfrak{L}_{\geqslant l+1}) \rightarrow H_{2}( \mathfrak{L}/\mathfrak{L}_{\geqslant m+1})
$$
is trivial for any $l > m$. Moreover, the linear map
$$
{t_{l,m}}_{\ast}: H_{3}( \mathfrak{L}/\mathfrak{L}_{\geqslant l+1}) \rightarrow H_{3}( \mathfrak{L}/\mathfrak{L}_{\geqslant m+1})
$$
is trivial for any $l \geqslant 2m$.
\end{lemma}
Let $\theta$ be a special expansion of $F_n$. As in \S5.2, for any $L \in SL_n(k+1)$, we have $Art^{\theta}(L)=(\exp(Y_1(L)),\ldots, \exp(Y_n(L)))$. We denote by $Y_i(L)^{(l)}$ the degree $l$-part of the Lie series $Y_i(L)$, i.e. we have
$$
Y_i(L)=\sum_{l\geqslant k+1}^{\infty} Y_i(L)^{(l)}.
$$
We then set
$$
\sigma_{L}:=\sum_{i=1}^{n}\sum_{l=k}^{2k+1}X_{i} \wedge Y_{i}^{(l)}(L) \in \Lambda^2(\mathfrak{L}/\mathfrak{L}_{\geqslant 2k+2})
$$
whose image under the boundary operator $\partial_2$ is given by
$$
\partial_2 \sigma_{L}=\sum_{i=1}^n \sum_{l=k}^{2k+1}[X_{i}, Y_{i}^{(l)}(L)] \in \Lambda^1 (\mathfrak{L}/\mathfrak{L}_{\geqslant 2k+2})
$$
By Lemma \ref{D}, we have $\partial_2 \sigma_{L}=0$. Therefore, $\sigma_L$ is a 2-cycle.

By lemma \ref{lem:pre}, the reduction $\{ \sigma_{L}\} \in \Lambda^2(\mathfrak{L}/\mathfrak{L}_{\geqslant 2k+1})$ is a 2-cycle and so is a boundary. We choose $t_L \in \Lambda^3(\mathfrak{L}/\mathfrak{L}_{\geqslant 2k+1})$ such that $\partial_3 t_L=\{ \sigma_{L}\}$. Since the reduction $\{ \sigma_L\}=0 \in \Lambda^2(\mathfrak{L}/\mathfrak{L}_{\geqslant k+1})$, the reduction $\{ t_L\} \in \Lambda^3(\mathfrak{L}/\mathfrak{L}_{\geqslant k+1})$ is a 3-cycle. We then define
$$
\overline{M}_{k+1}^{\theta}(L):=[\{t_L\}] \in H_3(\mathfrak{L}/\mathfrak{L}_{\geqslant k+1}).
$$
We note that the isomorphism $\theta : \mathfrak{m}(F_n/\Gamma_{l+1}F_n) \longrightarrow \mathfrak{L}/\mathfrak{L}_{\geqslant l+1}$ is compatible with the canonical projection $\mathfrak{L}/\mathfrak{L}_{\geqslant l+1} \rightarrow \mathfrak{L}/\mathfrak{L}_{\geqslant m+1}$ and $\mathfrak{m}(F_n/\Gamma_{l+1}F_n) \rightarrow \mathfrak{m}(F_n/\Gamma_{m+1}F_n)$ for all $l \geqslant m$. We set $M_{k+1}: SL_n(k+1) \longrightarrow H_3(\mathfrak{m}(F_n/\Gamma_{k+1}F_n))$ by the composition $M_{k+1}^{\theta}:= \theta_{\ast}^{-1}\circ \overline{M}_{k+1}^{\theta}$ where $\theta_{\ast}:H_3(\mathfrak{m}(F_n/\Gamma_{k+1}F_n)) \rightarrow H_3(\mathfrak{L}/\mathfrak{L}_{\geqslant k+1})$ is isomorphism induced by $\theta$.

\begin{lemma}
For $L \in SL_n(k+1)$, the map
$$
M_{k+1}^{\theta}: SL_n(k+1) \longrightarrow H_3(\mathfrak{m}(F_n/\Gamma_{k+1}F_n))
$$
is a well-defined monoid homomorphism. 
\end{lemma}

\begin{proof}
It is sufficient to show that $\overline{M}_{k+1}^{\theta}$ is a well-defined monoid homomorphism. 

To begin with, we show that $\overline{\mu}_{k+1}^{\theta}$ is independent of a choice of $t_L \in \Lambda^3 (\mathfrak{L}/\mathfrak{L}_{\geqslant 2k+1})$. Let $t'_L \in \Lambda^3(\mathfrak{L}/\mathfrak{L}_{\geqslant 2k+1})$ be another element such that $\partial_3 t'_L=\{\sigma_L\}$. Then the difference $t_L-t'_L$ is a 3-cycle. Then, by Lemma \ref{lem:pre}, the reduction $\{t_L-t'_L\}=\{t_L\}-\{t'_L\}$ must be null-homologous. Hence, we conclude that $[\{t_L\}]=[\{t'_L\}] \in H_3(\mathfrak{L}/\mathfrak{L}_{\geqslant k+1})$. Therefore, $\overline{M}_{k+1}^{\theta}$ is well-defined.

Secondly, we prove that $\overline{M}_{k+1}^{\theta}$ is a monoid homomorphism. Take another $L' \in SL_n(k+1)$ and choose $t_L' \in \Lambda^3(\mathfrak{L}/\mathfrak{L}_{\geqslant 2k+1})$ such that $\partial_3(t_L')=\{\sigma_{L'}\}$. Here, $\{\sigma_{L'}\}=\sum_{i=1}^n\sum_{l=k+1}^{2k}X_{i} \wedge {Y}_{i}^{(l)}(L')$. Then we set $\overline{\mu}_{k+1}^{\theta}(L')=[\{ t_{L'}\}]$. 

As in the proof of Proposition \ref{truncation}, we have 
$$
\log(\exp(Art^{\theta}(L)(Y_i(L')))\exp(Y_i(L))\equiv \sum_{l=k}^{2k}Y_i^{(l)}(L)+Y_i^{(l)}(L') \mod \mathfrak{L}_{2k+1}
$$

Hence, we see that the reduction $\{\sigma_{LL'}\}$ is given by
\begin{eqnarray*}
&\{\sigma_{LL'}\}&=\sum_{i=1}^n\sum_{l=k}^{2k}X_{i }\wedge (Y_{i}^{(l)}(L)+{Y}_{i}^{(l)}(L')) \\
&&=\{\sigma_L\}+\{\sigma_{L'}\} \in \Lambda^2(\mathfrak{L}/\mathfrak{L}_{\geqslant 2k+1})
\end{eqnarray*}
By choosing $t_{LL'} \in \Lambda^3(\mathfrak{L}/\mathfrak{L}_{\geqslant 2k+1})$ so that $\partial_3(t_{LL'})=\{\sigma_{LL'}\}=\{\sigma_L\}+\{\sigma_{L'}\} $, we have $\overline{M}_{k}^{\theta}(L\circ L')=[\{ t_{LL'}\}]$. Hence, 
\begin{eqnarray*}
\begin{split}
\overline{M}_{k+1}^{\theta}(L\circ L')&=[\{ t_{LL'}\}]\\
&=[\{ t_L\} + \{ t_{L'}\}]\\
&=\overline{M}_{k+1}^{\theta}(L)+\overline{M}_{k+1}^{\theta}(L').
\end{split}
\end{eqnarray*}
Therefore, we conclude that $\overline{M}_{k+1}^{\theta}$ is a monoid homomorphism and $M_{k+1}^{\theta}$ is also a monoid homomorphism.
\end{proof}
\begin{definition}
Let $\theta$ be a special expansion. For $L \in SL(k+1)$, the homomorphism
$$
M_{k+1}^{\theta}:SL_n(k+1)\longrightarrow H_3(\mathfrak{m}(F_n)/\Gamma_{k+1}\mathfrak{m}(F_n))
$$
is called the {\it infinitesimal Morita-Milnor homomorphism}.
\end{definition}
\begin{theorem}\label{thm:com}
Notations being as above, we have the following commutative diagram:
$$
\xymatrix{
SL_n(k+1) \ar[d]_{\mu_{[k+1,2k+1[}^{\theta}}  \ar[r]^{M_{k+1}^{\theta}} & H_3(\mathfrak{m}(F_n/\Gamma_{k+1}F_n)) \ar[r]^{\theta_{\ast}}_{\simeq} & H_3(\mathfrak{L}/\mathfrak{L}_{\geqslant k+1})\\
\bigoplus_{l=k+1}^{2k} \mathrm{D}_{l}(H) & & \ar[ll]^{\simeq}_{\eta} \bigoplus_{l=k+1}^{2k}\mathcal{C}^t_{l}(H) \ar[u]_{\Phi}^{\simeq}
}
$$
\end{theorem}
\begin{proof}
For $L \in SL_n(k+1)$, we want to show that $\overline{M}_{k+1}^{\theta}(L)=\Phi\circ \eta^{-1}\circ \mu_{[k,2k+1[}^{\theta}(L)$.
By the definition of the map $\mu_{[k+1,2k+1[}^{\theta}(L)$, we have
$$
\mu_{[k+1,2k+1[}^{\theta}(L)=\sum_{i=1}^{n}\sum_{l=k+1}^{2k}\left( X_{i} \otimes Y_{i}^{(l)}(L)\right).
$$
By the isomorphism $\eta:= \bigoplus_{l=k+1}^{2k} \eta_{l}: \bigoplus_{l=k+1}^{2k} \mathcal{C}^t_{l}(H) \overset{\simeq}{\longrightarrow} \bigoplus_{l=k+1}^{2k} \mathrm{D}_{l} (H) \subset \bigoplus_{l=k+1}^{2k} H\otimes \mathfrak{L}_{l}(H)$ we set $b_L:=\eta^{-1}\mu_{[k+1,2k+1[}^{\theta}(L)$ which is the diagrammatic description of $\mu_{[k+1,2k+1[}^{\theta}$. 

Let  $\phi(b_L) \in \Lambda^3(\mathfrak{L}/\mathfrak{L}_{\geqslant 2k+1})$ is the image under fission map $\phi$ of the linear combination of trees $b_L$. Then, it is enough to show that $\partial_3 \phi(b_L)=\{\sigma_L\}$. Here, $\partial_3 \phi(b_L)$ is given by
\begin{equation}\label{partial}
\partial_3\phi (b_L)=\sum_{v} \mathrm{col}(v) \wedge \mathrm{comm}(T_v)
\end{equation}
where the sum is taken over all univalent vertices $v$ of $b_L$ with $\mathrm{col}(v)$. Let us consider the natural embedding 
$$
\gamma: \bigoplus_{l=k}^{2k} H \otimes \mathfrak{L}_{l} \longrightarrow \Lambda^2(\mathfrak{L}/\mathfrak{L}_{\geqslant 2k+1}) ; u \otimes v \longmapsto \{u\} \wedge \{v\}.
$$
By the definition, we have
\begin{eqnarray}\label{gamma}
\begin{split}
\gamma\eta(b_L)&=\gamma\left(\sum_{v} \mathrm{col}(v) \otimes \mathrm{comm}(T_v)\right)\\
&=\sum_{v} \mathrm{col}(v) \wedge \mathrm{comm}(T_v)
\end{split}
\end{eqnarray}
where the range of sum is same as the above.

Hence, one sees that $\partial_3 \phi(b_L)=\gamma \eta(b_L)$ by (\ref{partial}) and (\ref{gamma}). 

Then, we have
\begin{eqnarray*}
\begin{split}
\partial_3 \phi(b_L)&=\gamma \eta(b_L)\\
&=\gamma\eta(\eta^{-1}\nu_{[k+1,2k+1[}^{\theta}(L))\\
&=\gamma\nu_{[k+1,2k+1[}^{\theta}(L)\\
&=\sum_{i=1}^n\sum_{l=k+1}^{2k}X_{i}\wedge Y_{i}^{(l)}(L)\\
&=\{\sigma_L\}.
\end{split}
\end{eqnarray*}
Therefore, we have $\overline{M}_k^{\theta}(L)=[\{\phi(b_L)\}]$. This completes the proof.
\end{proof}

\begin{corollary}
Notations being as above, we have the following commutative diagram:
$$
\xymatrix{
SL_n(k+1) \ar[rd]_{\mu_{k+1}}\ar[r]^{\mu_{k+1}^{\theta}} & H_3(\mathfrak{m}(F_n/\Gamma_{k+1}F_n)) \ar[d]_{d_{3,0}^2}\\
& H\otimes \mathfrak{L}_{k+1}(H)\\
}
$$
Here the homomorphism $d_{3,0}^2$ is the differential of the Hochschild-Serre spectral sequence associated to the central extension
$$
0\longrightarrow  \mathfrak{m}(\Gamma_{k}F_n/\Gamma_{k+1}F_n)  \longrightarrow \mathfrak{m}(F_n/\Gamma_{k+1}F_n)\longrightarrow  \mathfrak{m}(F_n/\Gamma_k F_n)\longrightarrow 1
$$
\end{corollary}
\begin{proof}

Let $\theta$ be a normalized expansion of $F_n$. Noting the Remark 3.2.1 (2) and (3), $\theta$ induces the isomorphism 
$$
\theta: \mathfrak{m}(F_n/\Gamma_{k+1}F_n)\longrightarrow \mathfrak{L}/\mathfrak{L}_{\geqslant k+1}
$$
and so  we have the following commutative diagram:

\begin{equation}\label{c}
\xymatrix{
0\ar[r] & \mathfrak{L}_k(H)  \ar@{=}[d] \ar[r] &\mathfrak{m}(F_n/\Gamma_{k+1}F_n)\ar[d]^{\theta}_{\simeq} \ar[r] & \mathfrak{m}(F_n/\Gamma_k F_n) \ar[d]^{\theta}_{\simeq}\ar[r] &1\\
0 \ar[r] & \mathfrak{L}_k(H) \ar[r] & \mathfrak{L}/\mathfrak{L}_{\geqslant k+1} \ar[r] & \mathfrak{L}/\mathfrak{L}_{\geqslant k} \ar[r] &1 
}
\end{equation}
By the naturality of Hochschild-Serre spectral sequence, it suffices to prove that the following diagram commutes:
$$
\xymatrix{
SL_n(k+1)\ar[rd]_{\mu_{k+1}} \ar[r]^{\overline{M}_{k+1}^{\theta}} & H_3(\mathfrak{L}/\mathfrak{L}_{\geqslant k+1}) \ar[d]^{d_{3,0}^2} \\
 & H\otimes \mathfrak{L}_{k+1}(H)
}
$$
where the differential $d_{3,0}^2$ is the differential of the Hochschild-Serre spectral sequence associated to the central extension of second line of (\ref{c}). By theorem \ref{thm:IgusaOrr} and theorem \ref{thm:Phi}, we have $d_{3,0}^2 \circ \overline{M}_{k+1}^{\theta}=\eta_{k+1}\circ \Phi^{-1}\circ \overline{M}_{k+1}^{\theta}$. And we have $\eta_{k+1}\circ \Phi^{-1}\circ \overline{M}_{k+1}^{\theta}=\mu_{k+1}^{\theta}$ by theorem \ref{thm:com}. Moreover, since $\mu_{k+1}^{\theta}$ is coincides with the Milnor invariant of degree $k+1$ by Theorem 5.2.2, we conclude that 
$$
d_{3,0}^2 \circ \overline{M}_{k+1}^{\theta}=\eta_{k+1}\circ \Phi^{-1}\circ \overline{M}_{k+1}^{\theta}=M_{k+1}^{\theta}=\mu_{k+1}.
$$
\end{proof}
Since the kernel of the truncation $\mu_{[k+1,2k+1[}$ is clearly $SL(2k+1)$, the following immediately follows.
\begin{corollary}
The kernel of $M_{k+1}^{\theta}$ is $SL_n(2k+1)$.
\end{corollary}


\subsection*{Acknowledgement}
The author would be sincerely grateful to  Gw\'ena\"el Massuyeau for reading draft of this paper and giving the author helpful comment and encouragement. The author would also be sincerely grateful to Hidekazu Furusho,Toshie Takata,  Kazuo Habiro for reading the draft and giving the author helpful comments. The author is partly supported by JSPS Fellows (14J12303).



\begin{thebibliography}{99} 
\bibitem[AET]{AET}
{ A. Alekseev, B. Enriquez,  C. Torossian},
{\it  Drinfeld associators, braid groups and explicit solutions of the Kashiwara-Vergne equations},
{Publ. Math. Inst. Hautes \'Etudes Sci. No. {\bf 112} (2010), 143--189.} 
\bibitem[AT]{AT}
{A.  Alekseev, C. Torossian} 
{\it The Kashiwara-Vergne conjecture and Drinfeld's associators}, Ann. of Math. (2) 175 (2012), no. 2, 415--463. 
\bibitem[BN]{BN}
{D. Bar-Natan},
{\it On the Vassiliev knot invariants},
{Topology 34, 1995, 101-126}.
\bibitem[Bi]{B}
{J. Birman},
{Braids, Links and Mapping class groups},
{Annals of mathematics studies 82, Princeton University Press,1975.}


 



\bibitem[HL]{HL}
{N. Habegger, X. S. Lin},
{\it On link concordance and Milnor's $\bar{\mu}$ invariants},
{Bull. London Math. Soc. 30 (1998), no. 4, 419--428. }

\bibitem[HM]{HabeM}
{N. Habegger, G. Masbaum},
{\it The Kontsevich integral and Milnor's invariants},
{ Topology 39 (2000), no. 6, 1253--1289. }.

\bibitem[IO]{IO}
{K. Igusa, K. E. Orr},
{\it Links, pictures and the homology of nilpotent groups},
{Topology, 40(6), 1125--1166, 2001}.


\bibitem[Ih]{Ih}
{Y. Ihara},
{\it The Galois representation arising from $\mathbb{P}^1-\{0,1,\infty\}$ and Tate twists of even degree}, {Galois groups over $\mathbb{K}$ (Berkeley, CA, 1987), 299--313, Math. Sci. Res. Inst. Publ., 16, Springer, New York, 1989. }




\bibitem[Ka]{Ka}
{N. Kawazumi},
{\it Cohomological aspects of Magnus expansions},
{arXiv: 0505497[math.GT], 2006}.
\bibitem[KMT]{KMT}
{H. Kodani, M. Morishita, Y. Terashima},
{\it Arithmetic topology in Ihara theory},
{in preparation}.





\bibitem[L]{L}
{X. S. Lin},
{\it Power series expansions and invariants of links},
{Geometric topology (Athens, GA, 1993), 184--202,
AMS/IP Stud. Adv. Math., 2.1, Amer. Math. Soc., Providence, RI, 1997.} 

\bibitem[MKS]{MKS}
{W. Magnus, A. Karrass, D. Solitar},
{\it Combinatorial Group Theory: Presentations of Groups in Terms of Generators and Relations},
{Second revised edition. Dover Publications, Inc., New York, 1976}.

\bibitem[Ma]{MA}
{G. Massuyeau},
{\it Infinitesimal Morita homomorphisms and the tree-level of the LMO invariant},
{Bull. Soc. Math. France 140 (2012), no. 1, 101--161.}

\bibitem[Ma2]{MA2}
{G. Massuyeau},
{\it Formal descriptions of Turaev's loop operations},
{arXiv:1511.03974}.

\bibitem[Mi]{Mi}
{J. Milnor},
{\it Isotopy of links}, 
{Algebraic geometry and topology, A symposium in honor of S.
Lefschetz, Princeton University Press, Princeton, NJ (1957), pp. 280--306.} 


\bibitem[Mo]{M2}
{S. Morita},
{\it Abelian quotients of subgroups of the mapping class group of surfaces},
{ Duke Math. J. 70 (1993), no. 3, 699--726. }

\bibitem[Q]{Q}
{D. Quillen},
{\it Rational homotopy theory},
{Ann. of Math. (2), 90, 205--295, 1969}.


\bibitem[S]{S}
{J. Stallings},
{\it Homology and central series of groups},
{J. Algebra, $\mathbf{2}$ (1965), 170--181.}

 \end{thebibliography}
\end{document}